\documentclass[12pt]{amsart}

\usepackage{amssymb}
\usepackage{amsmath}
\usepackage{amscd}
\usepackage{amsthm}
\usepackage{latexsym}

\newtheorem{thm}{Theorem}[section]
\newtheorem{prop}[thm]{Proposition}
\newtheorem{lem}[thm]{Lemma}
\newtheorem{cor}[thm]{Corollary}
\newtheorem{rem}[thm]{Remark}

\newtheorem{conjecture}[thm]{Conjecture}

\newcommand{\bs}{\boldsymbol}

\newcommand{\al}{\alpha}

\newcommand{\lam}{\lambda}

\newcommand{\vphi}{\varphi}

\newcommand{\lorarr}{\longrightarrow}

\newcommand{\bF}{\mathbb{F}}

\newcommand{\bZ}{\mathbb{Z}}

\newcommand{\mcF}{\mathcal{F}}

\newcommand{\mcP}{\mathcal{P}}
\newcommand{\mcQ}{\mathcal{Q}}

\DeclareMathOperator{\GL}{GL}
\DeclareMathOperator{\SL}{SL}
\DeclareMathOperator{\GU}{GU}
\DeclareMathOperator{\Cl}{Cl}
\DeclareMathOperator{\Irr}{Irr}

\title[Character Degree Product and 
Conjugacy Length Product]
{The Character Degree Product and 
the Conjugacy Length Product for Finite 
General Linear Groups}

\author{Akihiko Hida}
\author{Masahiro Sugimoto}

\address{Akhiko Hida, Faculty of Education, 
Saitama University,
Shimo-okubo 255, Sakura-ku, Saitama-city, Saitama, Japan}
\email{ahida@mail.saitama-u.ac.jp}
\address{Masahiro Sugimoto, 
Department of Mathematics, University of Tsukuba, 
Tennodai 1-1-1, Tsukuba, Ibaraki, Japan}
\email{sugimoto-m@math.tsukuba.ac.jp}

\begin{document}
\maketitle
\begin{abstract}
Let $G$ be a finite group. K. Harada conjectured that 
the product of degrees of all irreducible characters of $G$ 
divides the product of lengths of all conjugacy classes of $G$.  
We verify this conjecture for finite general linear 
groups and finite unitary groups. 
\end{abstract}

\section{Introduction}

Let $G$ be a finite group. 
Let $\Cl(G)$ be the set of conjugacy classes of $G$ and
$\Irr(G)$ the set of irreducible characters of $G$. 
It is well known that $|\Cl(G)|=|\Irr(G)|$ and 
that the length $|K|$ and the degree $\chi(1)$ both 
divide the order $|G|$ of $G$ for each $K \in \Cl(G)$, 
$\chi \in \Irr(G)$. 
In \cite{Ha}, K. Harada proposed a conjecture
concerning the product of class lengths and the 
product of characetr degrees. 

\begin{conjecture}[{\cite[Conjecture II]{Ha}}]\label{HC}
Let 
\[
h(G)=
\frac{\prod_{K \in \Cl(G)}|K|}
{\prod_{\chi \in \Irr(G)} \chi(1)}
\]
for a finite group $G$. Then $h(G)$ is an integer.
\end{conjecture}

It is clear that if $G$ is abelian then the conjecture holds 
since $|K|=\chi(1)=1$ for evrey 
$K \in \Cl(G)$ and $\chi \in \Irr(G)$. 
N. Chigira verified the conjecture for some finite 
simple groups including every sporadic simple groups.
The conjecture holds for symmetric groups and 
alternating groups \cite{Hi}. 
On the other hand, it has not been proved for 
solvable groups \cite[Conjecture 8.5]{N}. 
For a motivation and a background of the conjecture, see 
\cite[section 3]{Ha}. 
In this paper, we prove the conjecture for 
finite general linear groups $\GL_n(q)$ and finite unitary groups 
$\GU_n(q)$ over the field $\bF_q$ of 
$q$ elements. The following theorems are 
the main results of this paper.

\begin{thm}\label{HCGL}
Conjecture \ref{HC} holds for $\GL_n(q)$.  
\end{thm}

\begin{thm}\label{HCU}
Conjecture \ref{HC} holds for $\GU_n(q)$.  
\end{thm}

The conjugacy classes and irreducible characters of 
$\GL_n(q)$ are parametrized 
by partition valued functions $\bs{\lam}$ defined on 
the set of monic irreducible polynomials in $\bF_q[t]$ (other than $t$). 
Let $K_{\bs{\lam}}$ (resp. $d_{\bs{\lam}}$) be the 
conjugacy class (resp. the degree of the character) 
corresponding to $\bs{\lam}$. 
For a rational number $a=q^mr \in \bZ[q^{-1}]$ where 
$m,r \in \bZ$ and $(q,r)=1$, 
let $|a|_{q'}=r$ and we call $r$ a $q'$-part of $a$.   
Moreover, let 
$v_q$ be a valuation such that 
$v_q(q^mr)=m$. 
In \cite[Theorem 1.2]{S}, the second author proved that  
\[
\frac{|K_{\bs{\lam}}|_{q'}}{(d_{\bs{\lam}})_{q'}} 
\]
is an integer (see section 2 below). 
In particular, 
$h(G) \in \bZ[q^{-1}]$ and Conjecture \ref{HC} is equivalent to 
$v_q(h(\GL_n(q))) \geq 0$. 
In \cite[Theorem 1.3]{S}, it was proved that 
$v_q(h(\GL_n(q)))$ is a polynomial in $q$ with positive 
leading coefficient, in particular, 
Conjecture \ref{HC} holds for sufficiently large $q$. 
In this paper, we show that 
$v_q(h(\GL_n(q))) \geq 0$ for any $n$ and $q$.

In the proof of Theorem \ref{HCGL}, we need some combinatorial 
results on partitions of integers. 
For a partition $\lam=(\lam_1,\lam_2, \dots, \lam_l)$ 
of $n$, 
let 
\[n(\lam)=\sum_{i=1}^l(i-1)\lam_i.\]
We consider the sum $N(n)=\sum_{\lam \vdash n}n(\lam)$ 
of all values $n(\lam)$ $(\lam \vdash n)$. 
We need some upper bounds of $N(n)$. In particular 
we prove (Theorem \ref{t.Nn})
\[N(n) \leq \frac{1}{6}n(n+1)p(n)\]
where $p(n)$ is the number of partitions of $n$.  

In section 2, we state the main results more precisely. 
In section 3, we consider some upper bounds of 
$N(n)$. 
In section 4, we prove Theorem \ref{HCGL}. 
Finally, we prove Theorem \ref{HCU} in section 5.

%
%
%
%

\section{Conjugacy classes and irreducible characters of 
$\GL_n(q)$}

Let $\mcP$ be the set of all partitions of integers and  
$\mcP_n$ the set of all partitions of $n$. We set $p(n)=|\mcP_n|$. 

Let $G=\GL_n(q)$. Then 
$|G|=q^{\frac{1}{2}n(n-1)}\psi_n(q)$, where 
$\psi_m(t)=(t-1)(t^2-1)\cdots(t^m-1)$. 
Irreducible characters of $G$ were determined by Green \cite{G}. 
Let $\mcF_q$ be the set of monic irreducible polynomial over $\bF_q$ 
except for $t$. 
Let $\mcF_{q,n}$ be the set of polynomials in $\mcF_q$ of degree $n$.  
Let  
\[
M_n(\mcF_q, \mcP)=
\{\bs{\lam}: \mcF_q \lorarr \mcP~|~||\bs{\lam}||=n\}
\]
where $d(f)=\deg(f)$ and 
\[||\bs{\lam}||=\sum_{f \in \mcF_q}d(f)|\bs{\lam}(f)|.\]

Conjugacy classes of $G$ correspond to the functions in 
$M_n(\mcF_q, \mcP)$. The order of the centralizer of an element 
in the class corresponding to 
$\bs{\lam} \in M_n(\mcF_q, \mcP)$ is 
\[
\prod_{f \in \mcF_q}
a_{\bs{\lam}(f)}(q^{d(f)})
\]
where 
\[
a_{\lam}(q)=
q^{|\lam|+2n(\lam)-s(\lam)}
\prod_{i \geq 1}\psi_{m_i(\lam)}(q)
\]
and 
\[
s(\lam)=\sum_{i \geq 1}\frac{1}{2}m_i(\lam)
(m_i(\lam)+1)
\]
for a partition $\lam$ (\cite[IV (2.7)]{M}). 
Here, $m_i(\lam)$ is the multiplicity of $i$ in $\lam$, 
that is, $\lam=(1^{m_1}2^{m_2}\cdots)$. 
Let $K_{\bs{\lam}}$ be the 
conjugacy class corresponding to $\bs{\lam}$. Then 
\[
|K_{\bs{\lam}}|=
\frac{|G|}{\prod_{f \in \mcF_q}
a_{\bs{\lam}(f)}(q^{d(f)})}
=
\frac{q^{\frac{1}{2}n(n-1)}\psi_n(q)}
{\prod_{f \in \mcF_q}a_{\bs{\lam}(f)}(q^{d(f)})}.
\]

On the other hand, 
the degree of the irreducible character corresponding 
to $\bs{\lam}$ is 
\[
d_{\bs{\lam}}
=
\psi_n(q) 
\prod_{f \in \mcF_q} (q^{d(f)})^{n(\bs{\lam}(f))}
\tilde{H}_{\bs{\lam}(f)}(q^{d(f)})^{-1}
\]
where 
\[
\tilde{H}_{\lam}(q)=\prod_{x \in \lam}(q^{h(x)}-1)
\]
and $h(x)$ is the hook length of $\lam$ at $x$ 
for a partition $\lam$ (\cite[IV (6.7)]{M}). 
Here, we adopt $n(\bs{\lam}(f))$ instead of the 
conjugate partition $n(\bs{\lam}(f)')$. 
This is not essential since we consider the product of 
all character degrees. 
Here, the conjugate partition $\lam'$ of $\lam$ is the 
partition $\lam'=(\lam'_1\lam'_2 \cdots)$ 
where $\lam'_i$ is the number of positive integers $j$ such that $\lam_j \geq i$.
Now, we set 
\[
b_{\lam}(q)=q^{n(\lam)}
\frac{\prod_{i<j}(q^{\lam_i-\lam_j-i+j}-1)}
{\prod_{r=1}^{l(r)} \psi_{\lam_r+l(r)-r}(q)}
\]
where $l(\lam)$ is the length of a partition $\lam$. 
Then by \cite[p.10, Example 1]{M},
\[
b_{\lam}(q)=\frac{q^{n(\lam)}}{\tilde{H}_{\lam}(q)}
\]
and so 
\[
d_{\bs{\lam}}
=\psi_n(q)\prod_{f \in \mcF_q}b_{\bs{\lam}(f)}(q^{d(f)}).
\] 
This presentation of $d_{\bs{\lam}}$ coincides with the 
degree described in \cite[Theorem 14]{G} 
(and \cite[Lemma 2.2]{S}).

First we consider the $q'$-part, 
\[\frac
{|K_{\bs{\lam}}|_{q'}}
{(d_{\bs{\lam}})_{q'}}
=\prod_{f \in \mcF_q}
\frac
{\tilde{H}_{\bs{\lam}(f)}(q^{d(f)})}
{\prod_i\psi_{m_i(\bs{\lam}(f))}(q^{d(f)})}.
\]
For a partition $\lam$, let 
$V=\{x=(i,j) \in \lam~|~ (i,j+1) \not\in \lam\}$. 
Then 
\[
\prod_{i \geq 1}\psi_{m_i(\lam)}(q)
=
\prod_{x\in V}(q^{h(x)}-1) \mid 
\prod_{x \in \lam}(q^{h(x)}-1)=\tilde{H}_{\lam}(q)
\]
and 
\[\frac
{\tilde{H}_{\lam}(q)}
{\prod_i\psi_{m_i(\lam)}(q)}.
\]
is an integer. It follows that 
$\frac
{|K_{\bs{\lam}}|_{q'}}
{(d_{\bs{\lam}})_{q'}}$ is an integer. 

Next, we consider $q$-part. 
Following \cite{S}, we set  
\[\Omega(\bs{\lam})=
v_q
\Big(
\frac{|K_{\bs{\lam}}|}{(d_{\bs{\lam}})}
\Big)
=
v_q(|K_{\bs{\lam}}|)-
v_q(d_{\bs{\lam}}).
\]
Then 
\begin{eqnarray*}
\Omega({\bs{\lam}})
&=&
\begin{pmatrix} n \\ 2 \end{pmatrix}
-
\sum_{f \in \mcF_q} d(f)(|\bs{\lam}(f)|+2(n(\bs{\lam}(f)))-s(\bs{\lam}(f))+n(\bs{\lam}(f)))
\\ &=& 
\begin{pmatrix} n \\ 2 \end{pmatrix}
-
\sum_{f \in \mcF_q} d(f)(|\bs{\lam}(f)|+3(n(\bs{\lam}(f)))-s(\bs{\lam}(f))).
\end{eqnarray*}
To prove Theorem \ref{HCGL}, it suffices to show that 
\[
v_q(h(\GL_n(q)))=
\sum_{\bs{\lam} \in M_n(\mcF_q,\mcP)}
\Omega(\bs{\lam}) \geq 0.
\]
If $(n,q)=(2,2)$ then 
\[
h(\GL_2(2))=
\frac{1 \times 2 \times 3}{1 \times 1 \times 2}
=3
\]
and Conjecture \ref{HC} holds for $\GL_2(2)$. 
Hence Theorem \ref{HCGL} follows from the following theorem. 

\begin{thm}\label{MAIN1}
If $(n,q) \ne (2,2)$, then we have
\[
\sum_{\bs{\lam} \in M_n(\mcF_q,\mcP)}
\Omega(\bs{\lam}) \geq 
\begin{pmatrix}
n \\2
\end{pmatrix}.
\]
\end{thm}

In fact, we prove more precisely result (Theorem 
\ref{MAIN2}). 
Theorem \ref{MAIN1} is an consequence of Theorem \ref{MAIN2} below. 
Let $\bZ_{\geq 0}$ be the set of non negative integers. 
For a function $\al : \mcF_q \lorarr \bZ_{\geq 0}$, let \[
||\al||=\sum_{f \in \mcF_q} d(f)\al(f)
\]
\[
M_n(\mcF_q)=\{\al: \mcF_q \lorarr \bZ_{\geq 0}~|~
||\al||=n\}.
\]
For $\al \in M_n(\mcF_q)$, we set 
\[
M_{n, \al}(\mcF_q, \mcP)
=
\{\bs{\lam} \in M_n(\mcF_q,\mcP)~|~
|\bs{\lam}(f)|=\al(f) \ (\forall f \in \mcF_q)\}.
\]
Then
\[
M_n(\mcF_q,\mcP)=
\bigcup_{\al \in M_n(\mcF_q)} M_{n, \al}(\mcF_q,\mcP)
\]
and 
\[
M_{n, \al}(\mcF_q, \mcP) \lorarr 
\prod_{f \in \mcF_q} \mcP_{\al(f)}
\]
\[\bs{\lam} \mapsto (\bs{\lam}(f))_{f \in \mcF_q}
\]
is a bijection. Let 
\[
p(\al)=|M_{n, \al}(\mcF_q, \mcP)|=\prod_{f \in \mcF_q}p(\al(f)).
\]
If we fix $f \in \mcF_q$, then for any 
$\nu \in \mcP_{\al(f)}$,
\[
|\{\bs{\lam} \in M_{n,\al}(\mcF_q, \mcP)~|~
\bs{\lam}(f)=\nu\}|=p(\al)/p(\al(f))
\]
and this is independent of $\nu$. 
For a function $w: \mcP \lorarr \bZ_{\geq 0}$, we have 
\begin{eqnarray*}
\sum_{\bs{\lam} \in M_{n,\al}(\mcF_q, \mcP)}
\sum_{f \in \mcF_q} d(f)w(\bs{\lam}(f)) &=& 
\sum_{f \in \mcF_q} 
\sum_{\bs{\lam} \in M_{n,\al}(\mcF_q, \mcP)}
d(f)w(\bs{\lam}(f)) 
\\ &=& 
\sum_{f \in \mcF_q} 
d(f)\frac{p(\al)}{p(\al(f))}
(\sum_{\nu \in \mcP_{\al(f)}}w(\nu)).
\end{eqnarray*}

\begin{lem}\label{l.slam}
\[
\sum_{\lam \vdash n} s(\lam) 
=np(n)
\]
\[\sum_{\lam \vdash n} (|\lam|+3(n(\lam))-s(\lam))
=
3N(n)\]
\end{lem}

\begin{proof}
$\sum_{\lam \vdash n}s(\lam)$ 
equals to the sum of 
parts of all partitions of $n$, hence it equals to $np(n)$.
\end{proof}

For $\al \in M_n(\mcF_q)$, we have 
\begin{eqnarray*}
\sum_{\bs{\lam} \in M_{n,\al}(\mcF_q,\mcP)}
\Omega(\bs{\lam}) &=& 
\sum_{\bs{\lam} \in M_{n,\al}(\mcF_q,\mcP)}
\Big\{
\begin{pmatrix} n \\ 2 \end{pmatrix}
-
\sum_{f \in \mcF_q} d(f)(|\bs{\lam}(f)|+3(n(\bs{\lam}(f)))-s(\bs{\lam}(f)))
\Big\}
\\ &=&
p(\al)\begin{pmatrix} n \\ 2 \end{pmatrix}
-
3\sum_{f \in \mcF_q}(p(\al)/p(\al(f))d(f)N(\al(f))
\\ &=& 
p(\al)
\Big(
\begin{pmatrix} n \\ 2 \end{pmatrix}
-3\sum_{f \in \mcF_q}d(f)\bar{N}(\al(f))
\Big)
\end{eqnarray*}
where $\bar{N}(m)=N(m)/p(m)$.
Let 
\[
M_n^{(1)}(\mcF_q)=
\{\al \in M_n(\mcF_q)~|~
\al(f)=n \ (\exists f \in \mcF_{q,1})\}
\]
\[
M_n^{(n)}(\mcF_q)=
\{\al \in M_n(\mcF_q)~|~
\al(f)=1 \ (\exists f \in \mcF_{q,n})\}.
\]

The following is our main theorem. Theorem \ref{MAIN1} 
follows from Theorem \ref{MAIN2} immediately. 
We prove Theorem \ref{MAIN2} in section 4. 
 
\begin{thm}\label{MAIN2} 
(1) If $\al \not\in M^{(1)}_n(\mcF_q)$ then 
\[
\begin{pmatrix} n \\ 2 \end{pmatrix}
\geq
3\sum_{f \in \mcF_q}d(f)\bar{N}(\al(f)).
\]
In particular, 
\[
\sum_{\bs{\lam} \in M_{n,\al}(\mcF_q,\mcP)}\Omega(\bs{\lam})
\geq 0.
\]
(2) If $(n,q)\ne (2,2)$ then 
\[
\begin{pmatrix} n \\ 2 \end{pmatrix}
\Big(
\sum_{\al \in 
M_n^{(1)}(\mcF_q) \cup M_n^{(n)}(\mcF_q)}
p(\al)-1
\Big)
\geq 
3 \sum_{\al \in 
M_n^{(1)}(\mcF_q) \cup M_n^{(n)}(\mcF_q)}
\sum_{f \in \mcF_q}d(f)\bar{N}(\al(f))p(\al).
\]
In particular,
\[
\sum_{\al \in 
M_n^{(1)}(\mcF_q) \cup M_n^{(n)}(\mcF_q)}
\sum_{\bs{\lam} \in M_{n,\al}(\mcF_q,\mcP)}\Omega(\bs{\lam})
\geq 
\begin{pmatrix} n \\ 2 \end{pmatrix}.
\]
\end{thm}

%
%

\section{Upper bounds of $N(n)$}

Recall that $N(n)=\sum_{\lam \vdash n} n(\lam)$ and 
$n(\lam)=\sum_{i=1}^{l(\lam)}(i-1)\lam_i$ for a 
partiotion of $n$ where $l(\lam)$ is the length of $\lam$. 
In this section, we consider various upper bounds 
of $N(n)$. 
We set 
\begin{eqnarray*}
\mcP_{11}(n)
&=& \{\lam \vdash n ~|~  
\lam_{l(\lam)}=\lam'_{l(\lam')}=1\}
=
\{\lam \vdash n ~|~ \lam_{l(\lam)}=1, \ \lam_1>\lam_2\}
\\
\mcP_{21}(n)&=&\{\lam \vdash n ~|~  
\lam_{l(\lam)} \geq 2, \ \lam'_{l(\lam')}=1\}
=
\{\lam \vdash n ~|~  
\lam_{l(\lam)} \geq 2, \ \lam_1>\lam_2\}
\\
\mcP_{12}(n)&=&\{\lam \vdash n ~|~  
\lam_{l(\lam)}=1, \ \lam'_{l(\lam')}\geq 2\}
=
\{\lam \vdash n ~|~ 
\lam_{l(\lam)}=1, \ \lam_1=\lam_2\}
\\
\mcP_{22}(n)&=&\{\lam \vdash n ~|~  
\lam_{l(\lam)}, \lam'_{l(\lam')}\geq 2 \} 
=
\{\lam \vdash n ~|~ 
\lam_{l(\lam)}\geq 2, \ \lam_1=\lam_2\}
\end{eqnarray*}
and 
\[
p_{ij}(n)=|\mcP_{ij}(n)|, \ 
N_{ij}(n)=\sum_{\lam \in \mcP_{ij}(n)} n(\lam).
\]
Then 
\[\mcP_n=\bigcup_{i,j} \mcP_{ij}(n)\]
\[p(n)=\sum_{i,j}p_{ij}(n), \ 
N(n)=\sum_{i,j}N_{ij}(n).\]
Moreover, 
\[
\mcP_{ij}(n) \lorarr \mcP_{ji}(n)
\]
\[\lam \mapsto \lam'
\]
is a bijection and $p_{ij}(n)=p_{ji}(n)$. 

For $n \geq 3$, we define three bijections $\psi, \psi_{21}, \psi_{12}$ as follows.
\[
\psi: \mcP_{n-2} \lorarr \mcP_{11}(n)
\]
\[
\psi(\lam)=(\lam_1+1, \lam_2, \dots, \lam_{l(\lam)},1)\] 
\[
n(\psi(\lam))=n(\lam)+l(\lam)
\]

\[
\psi_{21}: \mcP_{21}(n-1) \cup \mcP_{22}(n-1) 
\lorarr 
\mcP_{21}(n)
\]
\[\psi_{21}(\lam)=
(\lam_1+1, \lam_2, \dots, \lam_{l(\lam)})
\]
\[
n(\psi_{21}(\lam))=n(\lam)
\]

\[
\psi_{12}: \mcP_{12}(n-1) \cup \mcP_{22}(n-1) 
\lorarr 
\mcP_{12}(n)
\]
\[\psi_{12}(\lam)=
(\lam_1, \lam_2, \dots, \lam_{l(\lam)},1)
\]
\[
n(\psi_{12}(\lam))=n(\lam)+l(\lam)
\]

\begin{lem}\label{l.N22}
(1) If $\lam \in \mcP_{22}(n)$ then 
$n(\lam)+n(\lam') \leq \frac{1}{4}n(n+1)$.
\\
(2) 
\[N_{22}(n) \leq \frac{1}{8}n(n+1)p_{22}(n)\]
\end{lem}

\begin{proof}
(1) Let $\lam=(\lam_1, \dots, \lam_l)$, $l=l(\lam) \geq 2$, 
$\lam_l=k \geq 2$. We proceed by induction on $l$.
If $l=2$, then $\lam=(k,k)$, $k \geq 2$, $2k=n$ and 
it follows that 
\[
n(\lam)+n(\lam')=k+k(k-1)=k^2=\frac{n^2}{4} \leq \frac{1}{4}n(n+1).
\]
Next, suppose that $l \geq 3$. 
Let $\mu=(\lam_1, \dots, \lam_{l-1})$. Then 
\[
n(\lam)=n(\mu)+k(l-1)
\]
\[
n(\lam')=n(\mu')+(1+2+\cdots+(k-1))
=n(\mu')+\frac{1}{2}k(k-1).
\]
Since $\lam_1=\lam_2$, $\lam_{l-1} \geq \lam_l \geq 2$, 
we have $\mu \in \mcP_{22}(n-k)$. Then by induction 
\begin{eqnarray*}
n(\lam)+n(\lam') &=& 
n(\mu)+n(\mu')+k(l-1)+\frac{1}{2}k(k-1) \\
& \leq &
\frac{1}{4}(n-k)(n-k+1)+(n-k)+\frac{1}{2}k(k-1) 
\\ 
&=& 
\frac{1}{4}
(n^2-(2k-5)n+(3k^2-7k)) 
\\ & \leq & 
\frac{1}{4}n(n+1). 
\end{eqnarray*}
(2) By (1), 
\begin{eqnarray*}
2N_{22}(n)&=&2 \sum _{\lam \in \mcP_{22}(n)}n(\lam) \\ 
&=&
\sum _{\lam \in \mcP_{22}(n)} (n(\lam)+n(\lam')) \\ 
& \leq &
\frac{1}{4}n(n+1)p_{22}(n). 
\end{eqnarray*}
\end{proof}

\begin{lem}\label{l.llam}
\[
\sum_{\lam \in \mcP_n} l(\lam) \leq 
\frac{1}{2}(n+1)p(n)
\]
\end{lem}

\begin{proof} 
If $\lam \vdash n$, then 
$l(\lam)+l(\lam')=l(\lam)+\lam_1 \leq n+1$ and hence 
\[2 \sum_{\lam \in \mcP_n}l(\lam)
=
\sum_{\lam \vdash n}(l(\lam)+l(\lam'))
\leq 
p(n)(n+1).
\]
\end{proof}

\begin{lem}\label{l.N11} 
For $n \geq 2$, we have the following.
\\
(1)
$N_{11}(n) \leq N(n-2)+\frac{1}{2}(n-1)p(n-2)$. 
\\
(2) If $N(n-2) \leq \frac{1}{6}(n-2)(n-1)p(n-2)$, then 
\[N_{11}(n) \leq \frac{1}{6}n(n+1)p_{11}(n).\]
\end{lem}

\begin{proof}
(1) If $n=2$ then $N_{11}(n)=0$ and the result follows.
Assume that $n \geq 3$. By the bijection 
$\psi: \mcP_{n-2} \lorarr \mcP_{11}(n)$, 
\begin{eqnarray*}
N_{11}(n)
&=& 
\sum_{\lam \vdash n-2}n(\psi(\lam)) 
\\ &=& 
\sum_{\lam \vdash n-2}(n(\lam)+l(\lam))
\\ &=& 
N(n-2)+\sum_{\lam \vdash n-2}l(\lam)
\\ & \leq & 
N(n-2)+\frac{1}{2}(n-1)p(n-2) 
\end{eqnarray*}
(2) If $n=2$ then $N_{11}(n)=0$ and the result follows. 
Assume that $n \geq 3$. Since $p(n-2)=p_{11}(n)$, 
it follows that 
\begin{eqnarray*}
N_{11}(n) & \leq & 
N(n-2)+\frac{1}{2}(n-1)p(n-2) 
\\ & \leq & 
\Big( \frac{1}{6}(n-2)(n-1)+\frac{1}{2}(n-1)\Big)p_{11}(n)
\\ &=& 
\frac{1}{6}(n-1)(n+1)p_{11}(n)
\\ & \leq & 
\frac{1}{6}n(n+1)p_{11}(n).
\end{eqnarray*}
\end{proof}

\begin{lem}\label{l.llam12}
For $n \geq 4$, 
\[
\sum_{\lam \in \mcP_{12}(n) \cup \mcP_{22}(n)} l(\lam) 
\leq 
\frac{2}{3}(n+1)
(p_{12}(n)+p_{22}(n)).
\]
\end{lem}

\begin{proof}
We set 
\begin{eqnarray*}
\mcQ&=&\mcP_{12}(n) \cup \mcP_{22}(n) 
\\
\mcQ_1 &=&\{\lam \in \mcQ~|~ 
l(\lam) >\frac{2}{3}(n+1)\}
\\
\mcQ_2&=&\{\lam \in \mcQ~|~ 
l(\lam) \leq \frac{2}{3}(n+1)\}
=\mcQ- \mcQ_1.
\end{eqnarray*}
We define a map $\vphi: \mcQ_1 \lorarr \mcQ_2$ as follows:
for $\lam \in \mcQ_1$, $\lam_1=\lam_2$, $l=l(\lam)$, 
\[
\vphi(\lam)=
\begin{cases}
(\frac{l}{2}, \frac{l}{2}, \lam_1-1, \lam_2-1, \dots, \lam_l-1)
& \mbox{($l$ : even)}\\
(\frac{l-1}{2}, \frac{l-1}{2}, \lam_1, \lam_2-1, \dots, 
\lam_l-1)
& \mbox{($l$ : odd)}.
\end{cases}
\]
We claim that 
$\vphi(\lam) \vdash n$ and $\vphi(\lam) \in \mcQ_2$.  
Assume that $l$ is even. Then 
$\lam_1-1 \leq n-l <\frac{l}{2}$ 
since $l>\frac{2}{3}(n+1)$. Hence $\vphi(\lam) \vdash n$. 
Moreover, 
\begin{eqnarray*}
l(\vphi(\lam)) &=& l-m_1(\lam)+2 
\\ & \leq & n-l+2
\\ &<& 
n-\frac{2}{3}(n+1)+2 
\\ & \leq & 
\frac{2}{3}(n+1)
\end{eqnarray*}
and $\vphi(\lam) \in \mcQ_2$. 
Next assume that $l$ is odd. 
Since $\frac{2}{3}(n+1)<l$, we have $2n+3 \leq 3l$. Hence  
\[\lam_1 \leq n-l+1 \leq \frac{l-1}{2}\]
and $\vphi(\lam) \vdash n$. 
If $\lam_1 \geq 2$, then as in the case that $l$ is even, 
\[
l(\vphi(\lam))=l-m_1(\lam)+2
\leq n-l+2 \leq \frac{2}{3}(n+1).
\]
On the other hand, if $\lam_1 =1$,  
\[\lam=(1^n), \ l=n\]
\[\vphi(\lam)=
(\frac{l-1}{2}, \frac{l-1}{2},1), \ l(\vphi(\lam))=3\]
\[3 <\frac{10}{3}=\frac{2}{3}(4+1) \leq \frac{2}{3}(n+1).\]
In either cases, we have that $\vphi(\lam) \in \mcQ_2$. 

Next we claim that $\vphi$ is injective. 
Let $\lam,\mu \in \mcQ_1$ and assume that 
$\vphi(\lam)=\vphi(\mu)$. 
If $l(\lam)$ is even and $l(\mu)$ is odd, then  
\[
\vphi(\lam)=
(\ast,\ast,\lam_1-1,\lam_2-1, \dots)
\]
\[
\vphi(\mu)=
(\ast,\ast,\mu_1,\mu_2-1, \dots)
\]
and we have a contradiction since $\lam_1=\lam_2$ and 
$\mu_1=\mu_2$. 
Hence we have that $l(\lam) \equiv l(\mu) \bmod 2$.
If $l(\lam) \equiv l(\mu) \equiv 0 \bmod 2$, then  
\[
\vphi(\lam)=
(\frac{l(\lam)}{2}, \frac{l(\lam)}{2}, \lam_1-1,\lam_2-1, 
\dots, \lam_{l(\lam)}-1)
\]
\[
\vphi(\mu)=
(\frac{l(\mu)}{2}, \frac{l(\mu)}{2},\mu_1-1,\mu_2-1, 
\dots, \mu_{l(\mu)}-1)
\]
and it follows that $l(\lam)=l(\mu)$, $\lam=\mu$.
Similarly, if 
$l(\lam) \equiv l(\mu) \equiv 1 \bmod 2$ then 
\[
\vphi(\lam)=
(\frac{l(\lam)-1}{2}, \frac{l(\lam)-1}{2}, \lam_1,\lam_2-1, 
\dots, \lam_{l(\lam)}-1)
\]
\[
\vphi(\mu)=
(\frac{l(\mu)-1}{2}, \frac{l(\mu)-1}{2},\mu_1,\mu_2-1, 
\dots, \mu_{l(\mu)}-1)
\]
and it follows that $l(\lam)=l(\mu)$, $\lam=\mu$.

Next we show that if $\lam \in \mcQ_1$ then  
$l(\lam)+l(\vphi(\lam)) \leq \frac{4}{3}(n+1)$. 
Let $l=l(\lam)$. 
If $l$ is even or $\lam_1 \geq 2$, then 
\[l(\vphi(\lam)) \leq n-l+2\]
and 
\[l+l(\vphi(\lam)) \leq n+2 \leq \frac{4}{3}(n+1).\]
On the other hand, if $l$ is odd and $\lam_1=1$, namely, 
$l=n$, $\lam=(1^n)$, then 
 \[l+l(\vphi(\lam))
=n+3 \leq \frac{4}{3}(n+1)\]
since $n \geq 5$, and $l(\vphi(\lam))=3$. 

Finaly, we prove the statement of the Lemma. 
Let $\mcQ_3=\mcQ_2-\vphi(\mcQ_1)$. Then 
\begin{eqnarray*}
\sum_{\lam \in \mcQ}l(\lam) 
&=& 
\sum_{\lam \in \mcQ_1}
(l(\lam)+l(\vphi(\lam)) + 
\sum_{\lam \in \mcQ_3}l(\lam)
\\ & \leq & 
\frac{4}{3}(n+1)|\mcQ_1|+\frac{2}{3}(n+1)|\mcQ_3|
\\ &=& 
\frac{2}{3}(n+1)|\mcQ|
\\ &=& 
\frac{2}{3}(n+1)
(p_{12}(n)+p_{22}(n)).
\end{eqnarray*}
This completes the proof of Lemma. 
\end{proof}

\begin{lem}\label{l.N12}
(1) For $n \geq 5$, 
\[
N_{21}(n)+N_{12}(n) \leq 
N_{21}(n-1)+N_{12}(n-1)+2N_{22}(n-1)+
\frac{2}{3}np_{12}(n).
\]
(2) For $n \geq 1$, 
\[
N_{21}(n)+N_{12}(n) \leq 
\frac{n}{3}(n+1)p_{12}(n).\]
\end{lem} 

\begin{proof} 
(1) 
\begin{eqnarray*}
N_{21}(n)+N_{12}(n) 
&=& 
\sum_{\lam \in \mcP_{21}(n-1) \cup \mcP_{22}(n-1)}
n(\lam) + 
\sum_{\lam \in \mcP_{12}(n-1) \cup \mcP_{22}(n-1)}
(n(\lam)+l(\lam))
\\ & \leq & 
N_{21}(n-1)+N_{22}(n-1)+N_{12}(n-1)+N_{22}(n-1)
\\ &&
+
\frac{2}{3}n(p_{12}(n-1)+p_{22}(n-1))
\\ &=& 
N_{21}(n-1)+N_{12}(n-1)+2N_{22}(n-1)+
\frac{2}{3}np_{12}(n) 
\end{eqnarray*}
by Lemma \ref{l.llam12} since $n-1 \geq 4$. 
\\
(2) Assume that $n \geq 5$. Then by induction, 
\begin{eqnarray*}
N_{21}(n)+N_{12}(n) &\leq& 
N_{21}(n-1)+N_{12}(n-1)+2N_{22}(n-1)+
\frac{2}{3}np_{12}(n) 
\\ &\leq& 
\frac{n}{3}(n-1)p_{12}(n-1)+\frac{2}{6}n(n-1)p_{22}(n-1)
+\frac{2}{3}np_{12}(n) 
\\ &=& 
\frac{1}{3}n(n-1)p_{12}(n)+\frac{2}{3}np_{12}(n)
\\ &=& 
\frac{1}{3}(n^2-n+2n)p_{12}(n) 
\\ &=& 
\frac{1}{3}n(n+1)p_{12}(n).
\end{eqnarray*}
For $1 \leq n \leq 4$, the results follows from the 
following table. 
\[
\begin{array}{c|cccc}
n& \mcP_{21}(n)& \mcP_{12}(n) &N_{21}(n)+N_{21}(n)&
\frac{1}{3}n(n+1)p_{12}(n) \\ \hline 
\\[-.4cm]
2 &(2)&(1^2)&0+1=1&2\\
3 &(3)&(1^3)&0+3=3&4\\
4 &(4)&(1^4)&0+6=6&\frac{20}{3}
\end{array}
\]
\end{proof}

\begin{thm}\label{t.Nn} 
\[
N(n) \leq \frac{1}{6}n(n+1)p(n)
\]
for $n \geq 1$. 
\end{thm}

\begin{proof} 
The result holds for $n=1,2$ since $N(1)=0$ and 
$N(2)=1$. Assume that $n \geq 3$. 
Then, by Lemma \ref{l.N22} and Lemma \ref{l.N12} 
\[N_{22}(n)\leq \frac{1}{6}n(n+1)p_{22}(n)\]
\[
N_{21}(n)+N_{12}(n) \leq 
\frac{1}{3}n(n+1)p_{12}(n).
\]
By induction, we may assume that
\[
N(n-2) \leq 
\frac{1}{6}(n-2)(n-1)p(n-2).
\]
Then it follows 
\[
N_{11}(n) \leq \frac{1}{6}n(n+1)p_{11}(n)
\]
by Lemma \ref{l.N11}. By these inequalities, we have 
\begin{eqnarray*}
N(n)&=&N_{11}(n)+N_{21}(n)+N_{12}(n)+N_{22}(n) 
\\ &\leq& 
\frac{1}{6}n(n+1)
(p_{11}(n)+p_{12}(n)+p_{21}(n)+p_{22}(n))
\\ &=& 
\frac{1}{6}n(n+1)p(n).
\end{eqnarray*}
\end{proof}

\begin{lem}\label{l.phin}
For $n \geq 6$, 
\[
|\mcF_{q,n}|-1 \geq \frac{1}{n}(q^n-q^{n-1}).
\]
\end{lem}

\begin{proof}
Since 
\[
|\mcF_{q,n}|=\frac{1}{n}
\sum_{d \mid n}\mu(n/d)q^d
\]
the claim is equivalent to 
\[
q^{n-1}-n \geq \sum_{d \mid n, \ d<n}(-\mu(n/d))q^d.
\]
Since $d \leq \frac{n}{2}$ for any $d \mid n$, $d<n$, 
it is enough to show that 
\[
q^{n-1}-n \geq \sum_{i=0}^mq^i
\]
for $m \leq \frac{n}{2}$. 
Since $n \geq 6$, we have $m+1 \leq n-2$ and 
\[
q^{n-1}-q^{m+1} \geq q^{n-1}-q^{n-2} 
\geq 
q^{n-2} \geq 2^{n-2} \geq n-1.
\]
Hence 
\[
q^{n-1}-n \geq q^{m+1}-1 
\geq \sum_{i=0}^mq^i . 
\]
\end{proof}

\begin{lem}\label{l.pn}
For $n \geq 8$,  
\[p(n) \leq \frac{n-1}{n} 2^{n-3}.\]
\end{lem}

\begin{proof} 
If $n=8$, then  
\[
p(8)=22
<\frac{7}{8}\times 2^5=28
\]
and the result holds.
By induction, assume that 
$p(n) \leq \frac{n-1}{n} 2^{n-3}$. 
Then, 
\begin{eqnarray*}
p(n+1) & \leq & 2p(n) 
\\ &\leq &
\frac{n}{n+1} \times \frac{n+1}{n} \times 
 \frac{n-1}{n}2^{n-2}
\\ &\leq&
\frac{n+1}{n}2^{n-2}.
\end{eqnarray*}
\end{proof}

\begin{lem}\label{l.Nn-2}
For $n \geq 8$,
\[
N(n) \leq 
\frac{1}{6}n(n-1)
(p(n)+\frac{2^{n-2}}{n}
).
\]
\end{lem}

\begin{proof}
It suffices to show that 
\[
\frac{1}{6}n(n+1)p(n)
\leq 
\frac{1}{6}n(n-1)
(p(n)+\frac{2^{n-2}}{n})
\]
by Theorem \ref{t.Nn}. But this is equivalent to 
\[
p(n) \leq \frac{n-1}{n} 2^{n-3}
\]
and the result follows from the previous Lemma. 
\end{proof}

\begin{thm}\label{t.Nn-phi}
If $(n,q) \ne (2,2), (2,3), (3,2)$, then 
\[
N(n) \leq 
\frac{1}{6}n(n-1)
(p(n)+\frac{|\mcF_{q,n}|-1}{q}
).
\]
\end{thm}

\begin{proof}
If $n \geq 8$, then by Lemma \ref{l.phin}
\begin{eqnarray*}
\frac{2^{n-2}}{n}
\leq
\frac{q^{n-2}}{n}
\leq
\frac{q^{n-1}}{n}
\cdot 
\frac{q-1}{q}
\leq 
\frac{|\mcF_{q,n}|-1}{q-1}
\cdot 
\frac{q-1}{q}
=
\frac{|\mcF_{q,n}|-1}{q}
\end{eqnarray*}
and the result follows from Lemma \ref{l.Nn-2}. 
The case $n \leq 7$ follows from the following table:

\[
\begin{array}{c|cccc}
n & N(n) & 
\binom{n}{2} & p(n) &  
|\mcF_{q,n}| \\ \hline
2 & 1 & 1 & 2 &\frac{1}{2}(q^2-q) \\ 
3 & 4 & 3 & 3 &\frac{1}{3}(q^3-q) \\ 
4 & 12 & 6 & 5 &\frac{1}{4}(q^4-q^2) \\ 
5 & 26 & 10 & 7 & \frac{1}{5}(q^5-q) \\
6 & 57 & 15 & 11 &\frac{1}{6}(q^6-q^3-q^2+q) \\
7 & 103 & 21 & 15 & \frac{1}{7}(q^7-q)
\end{array}
\]
In fact, if $2 \leq n \leq 7$ and $(n,q) \ne (2,2), (2,3), (3,2)$, 
then $|\mcF_{q,n}| \geq q+1$, that is, 
\[
1 \leq \frac{|\mcF_{q,n}|-1}{q}.
\]
On the other hand, in these cases, the inequality
\[N(n) \leq \frac{1}{3}\binom{n}{2}(p(n)+1)\]
holds. 
\end{proof}

\begin{cor}\label{l.gu}
(1) If $(n,q) \ne (2,2)$, then 
\[N(n) \leq \frac{1}{6}n(n-1)(p(n)+
\frac{|\mcF_{q,n}|-1}{q-1}).\]
(2) Let $\mcF^U_{q,n}$ be the set of all $U$ irreducible 
polynomials (See section 5). 
If $(n,q) \ne (2,2), (2,3), (3,2)$ then 
\[N(n) \leq \frac{1}{6}n(n-1)(p(n)+
\frac{|\mcF_{q,n}^U|-1}{q}).\]
\end{cor}

\begin{proof}
(1) If $(n,q) \ne (2,3), (3,2)$ then the result follows from 
Theorem \ref{t.Nn-phi}. The cases $(n,q)=(2,3), (3,2)$ follows from the 
table above. 
\\
(2) We have 
\[
|\mcF_{q,n}^U|=
\begin{cases}
|\mcF_{q,n}|+1 & (n=1) \\ 
|\mcF_{q,n}|-1 & (n=2) \\ 
|\mcF_{q,n}|& (n>2) \\ 
\end{cases}.
\]
If $n \geq 3$ then the result follows from Theorem 
\ref{t.Nn-phi}. If $n=2$ and $(n,q) \ne (2,2), (2,3)$, then 
\[
\frac{|\mcF^U_{q,2}|-1}{q}=
\frac{\frac{1}{2}q(q-1)-2}{q}\geq 1
\]
and 
\[
N(2)=1 \leq \frac{1}{3}(2+
\frac{|\mcF^U_{q,2}|-1}{q}).
\]
\end{proof}

\begin{rem} 
{\rm
Here, 
we consider some possibilities to improve the inequalities above. 
If $n \geq 6$, then by Lemma \ref{l.phin}, 
\[
|\mcF_{q,n}|-1 \geq \frac{q^n-q^{n-1}}{n} \geq 
\frac{q^{n-1}}{n} \geq q.
\]
We expect that the inequality 
\[N(n) \leq \frac{1}{6}n(n-1)(p(n)+1)\]
holds for any $n$.  
On the other hand, in \cite[Conjecture 4.2]{S}, 
a stronger conjecture has been proposed, that is, 
\[N(n) \leq \frac{1}{6}n(n-1)p(n)\]
for any $n \geq 7$. We have not been able to prove these 
stronger inequalities. 
}
\end{rem}


\section{Proof of Theorem \ref{MAIN2}}
In this section, we prove Theorem \ref{MAIN2}. 

\begin{lem}\label{l.nchoose2}
Let $\mu \vdash n$. Assume that $l=l(\mu) \geq2$. 
Then  
\[
\begin{pmatrix}
n \\ 2
\end{pmatrix}
\geq 
\sum_{i=1}^l
\begin{pmatrix}
\mu_i+1 \\ 2
\end{pmatrix}
\]
unless $\mu=(n-1,1)$. 
\end{lem}

\begin{proof}
It is clear that the result holds for $n \leq 4$. 
Suppose that $n \geq 5$. 
There exists $\nu \vdash n-1$ such that 
$\nu_i=\mu_i$ or $\nu_i=\mu_i-1$ for any $i$ and 
$\nu \ne (n-2,1)$. Then by induction, 

\[
\begin{pmatrix}
n \\ 2
\end{pmatrix}
=
\begin{pmatrix}
n-1 \\ 2
\end{pmatrix}
+(n-1)
\geq 
\sum_{i=1}^l
\begin{pmatrix}
\nu_i+1 \\ 2
\end{pmatrix}
+(n-1)
\geq
\sum_{i=1}^l
\begin{pmatrix}
\mu_i+1 \\ 2
\end{pmatrix}.
\]
\end{proof}

\begin{lem}\label{l.Nmi} 
Suppose that $l \geq 2$. For positive integers $m_i>0$ 
$(1 \leq i \leq l)$ such that $\sum_{i=1}^l m_i=n$, 
we have 
\[
3\sum_{i=1}^l
\bar{N}(m_i) \leq 
\begin{pmatrix}
n \\ 2
\end{pmatrix}.
\]
\end{lem}

\begin{proof} 
If $(m_1,m_2) \ne (n-1,1), (1,n-1)$, then 
\[
3\sum_{i=1}^l
\bar{N}(m_i) \leq 
\sum_{i=1}^l
\begin{pmatrix}
m_i+1 \\ 2
\end{pmatrix}
\leq 
\begin{pmatrix}
n \\ 2
\end{pmatrix}
\]
by Lemma \ref{l.nchoose2}. If $l=2$ and $(m_1,m_2)=(n-1,1)$, then  
\[
3(\bar{N}(m_1)+\bar{N}(m_2))
=
3\bar{N}(m_1)
\leq 
\begin{pmatrix}
m_1+1 \\ 2
\end{pmatrix}
=
\begin{pmatrix}
n \\ 2
\end{pmatrix}
\]
since $N(m_2)=N(1)=0$. 
\end{proof}

\begin{cor}\label{c.dNmi} Let 
$l \geq 1$, $d_i, m_i >0$ $(1 \leq i \leq l)$, 
$\sum_{i=1}d_im_i=n$. 
Then we have 
\[
3\sum_{i=1}^ld_i \bar{N}(m_i) 
\leq 
\begin{pmatrix}
n \\ 2
\end{pmatrix}
\]
except for the case $l=1$, $d_1=1$, $m_1=n$.
\end{cor}

\vspace{.8cm}

\noindent
{\bf (Proof of Theorem \ref{MAIN2})} 
\\
(1) We set $\mathrm{supp} (\al)=\{f_1, \dots, f_l\}$.  
Then 
\[\sum_{i=1}^l d(f_i)\al(f_i)=n\]
and $(l,d(f_1), \al(f_1)) \ne (1,1,n)$ since
$\al \not\in M_n^{(1)}(\mcF_q)$. 
By Corollary \ref{c.dNmi}, 
\[
3\sum_{i=1}^ld(f_i) \bar{N}(\al(f_i)) 
\leq 
\begin{pmatrix}
n \\ 2
\end{pmatrix}.
\]
(2) If $\al \in M_n^{(1)}(\mcF_q)$, then  
$p(\al)=p(n)$, 
$|M_n^{(1)}(\mcF_q)|=|\mcF_{q,1}|=q-1$. 
On the other hand, if 
$\al \in M_n^{(n)}(\mcF_q)$, then  
$p(\al)=p(1)=1$, 
$|M_n^{(1)}(\mcF_q)|=|\mcF_{q,n}|$. Hence 

\begin{eqnarray*}
\begin{pmatrix}
n \\ 2
\end{pmatrix}
(\sum_{\al \in M_n^{(1)}(\mcF_q)}p(\al)+
\sum_{\al \in M_n^{(n)}(\mcF_q)}p(\al)-1)
&=&
\begin{pmatrix}
n \\ 2
\end{pmatrix}
((q-1)p(n)+|\mcF_{q,n}|-1)
\\ &\geq& 
3(q-1)N(n)
\end{eqnarray*}
by Corollary \ref{l.gu}(1) since $(n,q) \ne (2,2)$. 
On the other hand, 
\[
\sum_{\al \in M_n^{(1)}(\mcF_q)}
(\sum_{f \in \mcF_q}d(f)\bar{N}(\al(f)))
p(\al)
=
\sum_{\al \in M_n^{(1)}(\mcF_q)}
\bar{N}(n)p(n)
=(q-1)N(n)
\]
\[
\sum_{\al \in M_n^{(n)}(\mcF_q)}
(\sum_{f \in \mcF_q}d(f)\bar{N}(\al(f)))
p(\al)
=
\sum_{\al \in M_n^{(n)}(\mcF_q)}
n \bar{N}(1)p(1)=0.
\]
Hence it follows 
\[
\sum_{\al \in M_n^{(1)}(\mcF_q)\cup M_n^{(n)}(\mcF_q)}
(\sum_{f \in \mcF_q}d(f)\bar{N}(\al(f)))
p(\al)
=(q-1)N(n)
\]
and this completes the proof.

\vspace{.8cm}

We add a property of the integer $h(G)$. 
N. Chigira conjectured 
that $|G'|$ divides $h(G)$ for any finite group $G$ 
where $G'$ is the commutator group of $G$ 
(cf.\cite[p.385, Remark]{Ha}).
We prove that this property holds for $G=\GL_n(q)$. 

\begin{prop}\label{p.commutator}
Let $G=\GL_n(q)$. Then $|G'|$ divides $h(G)$. 
\end{prop}

\begin{proof}
If $(n,q)=(2,2)$ then $h(G)=3=|G'|$. 
So we may assume that $(n,q) \ne (2,2)$. 
Then $G'=\SL_n(q)$ and $|G'|=|G|/(q-1)$. 
By Theorem \ref{MAIN1},  
\[v_q(h(G)) \geq 
\begin{pmatrix}
n \\2 
\end{pmatrix}
=v_q(|G|)=v_q(|G'|). 
\]
Let $\bs{\lam}: f \mapsto \lam=(n)$ for some $f \in \mcF_{q,1}$. 
Then 
\[|K_{\bs{\lam}}|_{q'}/(d_{\bs{\lam}})_{q'}=
\frac{\psi_n(q)}{q-1}=|G'|_{q'}\]
and so 
$|G'|_{q'}$ divides $h(G)$. 
\end{proof}


\section{Unitary groups}

For a monic polynomial
\[f(t)= t^d+a_{d-1}t^{d-1}+\cdots+a_0\]
over $\bF_{q^2}$ with $a_0\neq 0$, we denote
\[\tilde{f}(t)= a_0^{-q}(a_0^q t^d+a_1^q t^{d-1}+\cdots+1).\]
We call a monic polynomial $f(t)$ U-irreducible if $f(t)$ is irreducible and $f(t)=\tilde{f}(t)$, or $f(t)=g(t)\tilde{g}(t)$, where $g(t)$ is irreducible and $g(t)\neq\tilde{g}(t)$.
Conjugacy classes and irreducible characters of $\GU_n(q)$ are parametrized by maps from the set $\mcF_{q}^{U}$ of monic U-irreducible polynomials excluding $f(t)=t$ to the set $\mcP$ of partitions.
By a theorem of Wall \cite{wall}, there exists a bijection from the set of maps $\boldsymbol{\lambda}: \mcF_{q}^{U} \to \mcP$ satisfying $\Vert\boldsymbol{\lambda}\Vert =n$ to the conjugacy classes of $\GU_n(q)$.

\begin{lem}
  The size of the conjugacy class $K_{\boldsymbol{\lambda}}^U$ corresponding to a map $\boldsymbol{\lambda}\in M_n(\mcF_{q}^{U})$ is
  \[\frac{|\GU_n(q)|}{(-1)^n a_U({\boldsymbol{\lambda}})},\]
  where
  \[a_U({\boldsymbol{\lambda}})= \prod_{f\in \mcF_q^{U}}a_{\boldsymbol{\lambda}(f)}((-q)^{\deg(f)}).\]
\end{lem}

V. Ennola defined, for each map $\boldsymbol{\lambda}\in M_n(\mcF_{q}^{U})$, an ``irreducible C-function $\chi_{\boldsymbol{\lambda}}$'' and showed the irreducible C-functions form an orthonormal basis for the vector space of class functions on $\GU_n(q)$ \cite[Theorem 1]{ennola}.
N. Kawanaka \cite{kawanaka} proved Ennola's conjecture, that the irreducible C-functions are the irreducible 
characters of $\GU_n(q)$. 

We define the sets 
$M_n(\mcF_q^U)$,  
$M_n^{(1)}(\mcF_q^U)$, $M_n^{(n)}(\mcF_q^U)$ and  
$M_{n,\al}(\mcF_q^U,\mcP)$ in the same way as in 
section 2, replacing $\mcF_q$ to $\mcF_q^U$.

\begin{lem}[Ennola duality \cite{ennola,kawanaka}]
  The degree $d_{\bs{\lam}}^U$ of the irreducible 
character $\chi_{\boldsymbol{\lambda}}$ corresponding 
to a map $\boldsymbol{\lambda}\in M_n(\mcF_{q}^{U})$ is 
  \[|\psi_n(-q)\prod_{f\in \mcF_{q}^{U}} b_{\boldsymbol{\lambda}(f)}((-q)^{\deg(f)})|.\]
\end{lem}

As in the case of $\GL_n(q)$, 
$\dfrac
{|K_{\bs{\lam}}^U|_{q'}}
{(d_{\bs{\lam}}^U)_{q'}}$ is an integer. 
Let 
\[\Omega^U(\bs{\lam})=
v_q \Big(
\frac{|K_{\bs{\lam}}^U|}{d_{\bs{\lam}}^U}
\Big).
\]
Then, similary to Theorem \ref{MAIN2}, we have the 
following.

\begin{thm} 
If $\al \not\in M^{(1)}_n(\mcF_q^U)$ then 
\[
\sum_{\bs{\lam} \in M_{n,\al}(\mcF_q^U,\mcP)}
\Omega^U(\bs{\lam})
\geq 0.
\]
On the other hand, 
if $(n,q)\ne (2,2), (2,3), (3,2)$ then 
\[
\sum_{\al \in 
M_n^{(1)}(\mcF_q^U) \cup M_n^{(n)}(\mcF_q^U)}
\sum_{\bs{\lam} \in M_{n,\al}(\mcF_q^U,\mcP)}
\Omega^U(\bs{\lam})
\geq 
\begin{pmatrix} n \\ 2 \end{pmatrix}.
\]
In particular, 
\[
v_q(h(\GU_n(q))=
\sum_{\bs{\lam}\in M_n(\mcF_q^U, \mcP)}\Omega^U(\bs{\lam})
\geq 
\begin{pmatrix}n\\2\end{pmatrix}.
\]
for $(n,q) \ne (2,2), (2,3), (3,2)$.
\end{thm} 

Here we use Corollary \ref{l.gu}(2) instead of Corollary 
\ref{l.gu}(1). Theorem \ref{HCU} follows immediately from this result. Moreover, $|\GU_n(q)'|$ divides $h(\GU_n(q))$ as in 
Proposition \ref{p.commutator}. 
The small cases follow
from the following calculation by GAP \cite{gap}. 
Here, $G=\GU_n(q)$ and $h'(G)=h(G)/|G'|$ and 
$G'$ is the commutator group of $G$. 

\begin{table}[htb]
\begin{tabular}{ccccccc}
$(n,q)$ & smallGroup ID& $|G|$ & $|G'|$ & $h(G)$ & $h'(G)$ 
\\ \hline
$(2,2)$ & $(18,3)$ & $18$ & $3$ & $27$ & $9$  \\
$(2,3)$ & $(96,67)$ & 
$96$ & $24$ & $2^{12}3^4$ & $2^9 3^3$  \\
$(3,2)$ & $(648,533)$& $648$ & $216$ & 
$2^{18}3^{21}$ & 
$2^{15}3^{18}$   \\
\end{tabular}
\end{table}

\begin{center}
\end{center}


\end{document}